\documentclass[a4paper,12pt,reqno]{article}
\usepackage{amssymb,hyperref}
\usepackage{amsthm}
\usepackage{amsmath,wrapfig,sidecap,float}
\usepackage{latexsym}
\usepackage[ansinew]{inputenc}
\usepackage[dvips]{graphicx}
\usepackage{color}
\usepackage{caption}
\newtheorem{theorem}{Theorem}
\newtheorem{lemma}[theorem]{Lemma}
\newtheorem{cor}[theorem]{Corollary}
\theoremstyle{definition}
\newtheorem{definition}[theorem]{Definition}
\newtheorem{example}[theorem]{Example}

\title{Permutations with a distinct divisor property}
 
\author{Mohammad Javaheri, Nikolai A. Krylov\\ ~ \\
Siena College, Department of Mathematics\\
515 Loudon Road, Loudonville NY 12211, USA\\ ~ \\
mjavaheri@siena.edu, nkrylov@siena.edu}
\date{}

\begin{document}
\maketitle
\begin{abstract}
A finite group of order $n$ is said to have the distinct divisor property (DDP) if there exists a permutation $g_1,\ldots, g_n$ of its elements such that $g_i^{-1}g_{i+1} \neq g_j^{-1}g_{j+1}$ for all $1\leq i<j<n$. We show that an abelian group is DDP if and only if it has a unique element of order 2. We also describe a construction of DDP groups via group extensions by abelian groups and show that there exist infinitely many non abelian DDP groups.
\end{abstract}



\section{Introduction}\label{intro}

A Costas array of order $n$ is a permutation $x_1,\ldots, x_n$ of $\{1,2,\ldots, n\}$ such that the ${n \choose 2}$ vectors $(j-i, x_j-x_i)$, $i\neq j$, are all distinct. Costas arrays were first studied by John P. Costas for their applications in sonar and radar \cite{Costas1, Costas2}. Several algebraic constructions of Costas arrays exist for special orders $n$, such as Welch, Logarithmic Welch, and Lempel constructions \cite{Golomb1, Golomb2, Golomb3}. Through exhaustive computer searches, all Costas arrays of order $n\leq 29$ have been found \cite{DISJ}. However, the problem of finding Costas arrays for larger orders becomes computationally very difficult. The weaker notion of {\it DDP permutation} requires only the consecutive \emph{distinct difference property} i.e., $x_{i+1}-x_{i} \neq x_{j+1}-x_j$ for all $1\leq i<j <n$. By recursive constructions, an abundance of DDP permutations can be found, at least $2^n$, of order $n$ \cite{BattenSane}. 

In this paper, we are interested in a notion slightly stronger than DDP. 

\begin{definition}\label{modular}
A DDP sequence mod a positive integer $n$ is a permutation $x_0,\ldots, x_{n-1}$ of the elements of $\mathbb{Z}_n=\mathbb{Z}/n\mathbb{Z}$ such that $x_0=0$ and $x_{i+1}-x_i \not \equiv x_{j+1}-x_j \pmod n$ for all $0\leq i <j <n-1$. 
\end{definition}

The first example of a DDP sequence mod 12 was introduced by F. H. Klein in 1925 as the all-interval twelve-tone row, series, or chord
$$
F,E,C,A,G,D,A\flat, D\flat, E \flat, G \flat, B \flat, C \flat,
$$
named the \emph{Mutterakkord} (mother chord) \cite{Klein}. In integers mod 12, this sequence reads
$$
0, 11, 7, 4, 2, 9, 3, 8, 10, 1, 5, 6,
$$
and the sequence of consecutive differences mod 12 is given by 11, 8, 9, 10,7, 6, 5, 2, 3, 4, 1, which are all distinct. By 1952, there were 18 known examples of all-interval series \cite{Eimert}. In 1965, IBM 7094 listed all of the 3856 examples of all-interval rows \cite{Ferentz}. Another example of 
an eleven-interval, twelve-tone row is the {\it Grandmother} chord, invented by Nicolas Slonimsky in 1938 \cite{Slonimsky}. 

\begin{figure}[H]
\centering
\includegraphics[scale=0.3]{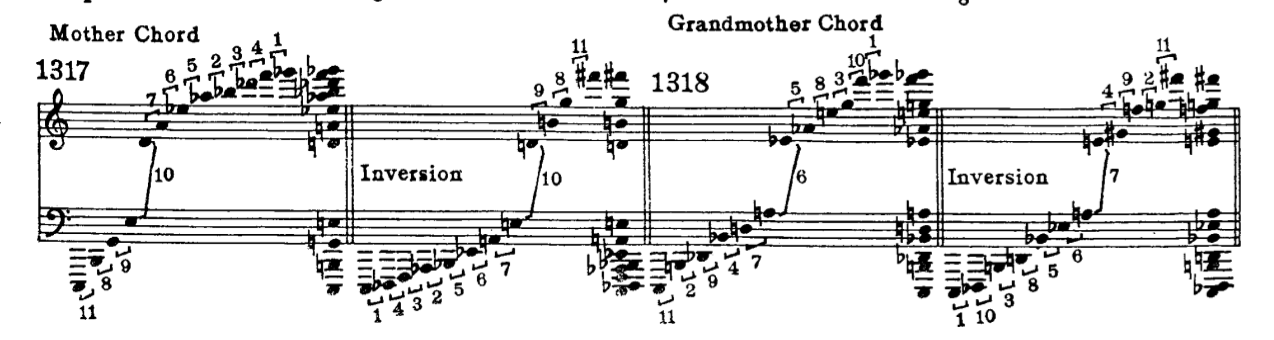}
\caption{An image of the Mother chord and Grandmother chord in Slonimsky's \emph{Thesaurus of Scales and Melodic Patterns} (p.\ 185).}
\end{figure}

 The grandmother chord has the additional property that the intervals are odd 
and even alternately, and the odd intervals decrease by one whole-tone, while the even intervals increase by one whole-tone. In integers mod 12, the grandmother chord is 
$$
0, 11, 1, 10, 2, 9, 3, 8, 4, 7, 5, 6,
$$
where the sequence of consecutive differences mod 12 is given by
11, 2, 9, 4, 7, 6, 5, 8, 3, 10, 1. Inspired by Slonimsky's grandmother chord, we define the Slonimsky sequence modulo $n$ by letting 
\begin{equation}\label{slonseq}
s_i=(-1)^{i} \lceil i/2 \rceil=\begin{cases} i/2 & \mbox{if $i$ is even;} \\ n-(i+1)/2 & \mbox{if $i$ is odd}.\end{cases}
\end{equation}

Then the sequence $s_0,\ldots, s_{n-1}$ is a DDP sequence modulo $n$ if and only if $n$ is even. 
If $x_0,\ldots, x_{n-1}$ is a DDP sequence modulo $n$, then the sequence $rx_0,\ldots, rx_{n-1}$ is also a DDP sequence modulo $n$ for each $r$ with $\gcd(r,n)=1$. Therefore, there are at least $\phi(n)$ DDP sequences mod an even integer $n$. The numbers of DDP sequences mod even integers are given by the sequence \cite{Gilbert,OEIS}
$$\href{https://oeis.org/A141599}{A141599:}~ 1, 2, 4, 24, 288, 3856, 89328, 2755968, 103653120, \ldots .$$ 
There are no DDP sequences mod odd $n$ (see Lemma \ref{none}). 

In our next definition, we generalize Definition \ref{modular}, which pertains to the group $(\mathbb{Z}_{n},+)$, to any finite group $G$. 

\begin{definition}
Let $G$ be a finite group with $n$ elements. We say a permutation $g_0,\ldots, g_{n-1}$ of elements of $G$ has the {\sl distinct divisor property} (DDP) or $g_0,\ldots, g_{n-1}$ is a DDP sequence, if $g_0=1_G$ and $g_i^{-1}g_{i+1} \neq g_j^{-1}g_{j+1}$ for all $0\leq i<j < n-1$. The set of all DDP sequences in $G$ is denoted by ${\cal O}_G$. We say $G$ is a DDP group if ${\cal O}_G \neq \emptyset$. 
\end{definition}

For odd values of $n$, instead of distinct consecutive differences, the sequence \eqref{slonseq} has distinct consecutive \emph{signed} differences. This motivates the following definition.

\begin{definition}\label{slondef}
Let $p_0,\ldots, p_{n-1}$ be a permutation of elements of an abelian group $G$ with $p_0=0$. The sequence of signed differences is defined by $h_0=0$ and $h_i=(-1)^{i-1}(p_{i-1}-p_i)$ for $1\leq i < n$. We say $p_0,\ldots, p_{n-1}$ is a {\it Slonimsky} sequence if the following conditions hold:
\begin{itemize}
\item[i)] $h_i \neq h_j$ for all $0\leq i <j  < n$.
\item[ii)] $h_i+h_{n-i}=0$ for all $0< i < n$.
\item[iii)] $p_i +p_{n-i-1} =p_{n-1}$ for all $0\leq i< n$, where we refer to $p_{n-1}$ as the last term of the sequence.
\end{itemize}
\end{definition}

For example, the following sequence is a Slonimsky sequence in $\mathbb{Z}_7$:
$$0, 6, 1, 5, 2, 4, 3,$$
and its sequence of signed differences is 0, 1, 2, 3, 4, 5, 6. Slonimsky sequences in odd abelian groups play an important role in constructing DDP sequences via group extensions, and we study 
them in section \ref{Slonimsky}.

This is how this paper is organized. In section \ref{Slonimsky}, we show that every odd abelian group has a Slonimsky sequence. In section \ref{extensions}, we use the existence of Slonimsky sequences in odd abelian groups to show that every central extension of an even DDP group by an odd abelian group is DDP (see Cor.\ \ref{centralext}). We also show that for every odd nilpotent group $G$ and an even DDP group $K$, the direct product $G \times K$ is DDP (see Theorem \ref{nil2}). In particular, $G \times \mathbb{Z}_{2^m}$ is DDP for every odd nilpotent group $G$ and every integer $m\geq 1$. 

In section \ref{abelian}, we show that a finite abelian group is DDP if and only if it has a unique element of order 2. We also find a lower bound on the number of DDP sequences in an abelian group $G$ in terms of the prime factorization of its order. In particular, 
we will show that if $n=2^m kl$ for $m\geq 1$ and relatively prime odd integers $k,l$, then there are at least $(2k)^{l-1}$ DDP sequences modulo the even integer $n$ (see Cor.\ \ref{sizeo}). Finally, in section \ref{nonabelian},  we will show that there are infinitely many non abelian DDP groups. 

\section{Slonimsky sequences in abelian groups}\label{Slonimsky}

In this section, we prove that every abelian odd group has a Slonimsky sequence. This result will only be needed in the proof of Theorem \ref{extconst} and can be skipped in a first reading. We begin with the cyclic case. 

\begin{lemma}\label{agaf}
If $n$ is odd, then $G=(\mathbb{Z}_n,+)$ has a Slonimsky sequence with the last term $(n-1)/2$. 
\end{lemma}

\begin{proof}
Let $p_i=(-1)^{i} \lceil i/2 \rceil \mod n$ for $0\leq i \leq n-1$. Then, for $1\leq i \leq n-1$, we have
\begin{align}\nonumber
h_i  =(-1)^{i-1}(p_{i-1}-p_i) & =(-1)^{i-1} \left ( (-1)^{i-1}\lceil (i-1)/2 \rceil - (-1)^{i} \lceil i/2 \rceil \right ) \\ \nonumber
 & =\lceil (i-1)/2 \rceil + \lceil i/2 \rceil = i,
 \end{align}
 hence property (i) in Definition \ref{slondef} holds. Moreover, $h_i+h_{n-i}=i+n-i=0 \pmod n$ and $p_i+p_{n-i-1}=(-1)^{i}\lceil i/2 \rceil +(-1)^{n-i-1} \lceil (n-i-1)/2 \rceil=(n-1)/2$ whether $i$ is even or odd. It follows that $p_0,\ldots, p_{n-1}$ is a Slonimsky sequence.
\end{proof}

\begin{theorem}\label{spddp}
Let $G=\mathbb{Z}_{m_1}\times \cdots \times \mathbb{Z}_{m_d}$ be an odd abelian group. Then there exists a Slonimsky sequence in $G$ with the last term
$$((m_1-1)/2,\ldots, (m_d-1)/2).$$
\end{theorem}

\begin{proof}
Proof is by induction on $d$. The claim for $d=1$ follows from Lemma \ref{agaf}. For $d>1$, let $H=\mathbb{Z}_{m_1} \times \cdots \times \mathbb{Z}_{m_{d-1}}$ and $m_d=m=2l-1$. By the inductive hypothesis for $H$, there exists a Slonimsky sequence $p_0,\ldots, p_{n-1}$ in $H$ with signed differences $h_0,\ldots, h_{n-1}$ such that 
\begin{align}
h_i+h_{n-i} & =0,~ \forall i \in \{1,\ldots, n-1\}; \\
 p_i +p_{n-i-1} & =((m_1-1)/2,\ldots, (m_{d-1}-1)/2),~ \forall i \in \{0,\ldots, n-1\}.
\end{align}
In order to define the Slonimsky sequence $P_0,\ldots, P_{mn-1}$ in $G=H\times \mathbb{Z}_m$, we first define its sequence of signed differences $g_i$, $1\leq i \leq mn$, in $G$ as follows. For $1\leq i \leq mn$, write $i=qn+r$, where $0\leq q \leq m-1$ and $0\leq r \leq n-1$. If $r=0$, we let $g_i=(0_H,q) \in H \times \mathbb{Z}_m$, and if $0<r\leq m-1$, we let 
$$g_i=\left (h_r,(-1)^{q}l+2\lceil q/2 \rceil \right).$$
We first show that $g_0,\ldots, g_{mn-1}$ is a permutation of elements of $G$. Suppose $g_i = g_j$, where $i=qn+r$ and $j=pn+t$. If $r=0$, then $g_j=g_i=(0_H,q)$ which implies that $t=0$, hence $g_j=(0_H,p)$, and so $p=q \Rightarrow i=j$. Thus, suppose that $r,t \neq 0$. It follows from $g_i=g_j$ that $h_r=h_t$, and so $r=t$. It also follows from $g_i=g_j$ that $(-1)^ql+2 \lceil q/2 \rceil=(-1)^pl+2 \lceil p/2 \rceil$. If $p-q$ is odd, we conclude that $|2\lceil q/2 \rceil - 2\lceil p/2 \rceil |=2l$, which is a contradiction, since $2\lceil p/2 \rceil, 2\lceil q/2 \rceil \in \{0, 2, \ldots, 2l-2\}$. If $p-q$ is even, we conclude that $2\lceil q/2 \rceil = 2\lceil p/2 \rceil$, which implies that $p=q \Rightarrow i=j$. Therefore, $g_0,\ldots, g_{mn-1}$ is a permutation of elements of $G$.

Next, we define $P_i=\sum_{k=0}^i (-1)^{k}g_k$ and show that $P_0,\ldots, P_{mn-1}$ is a Slonimsky sequence in $G$. A simple induction shows that for $i=qn+r$ with $0\leq q \leq m-1$ and $0\leq r \leq n-1$, we have
$$
P_i=\begin{cases} 
(p_r,q/2) & \mbox{if $q$ is even and $r$ is even;} \\
(p_r,-l-q/2)& \mbox{if $q$ is even and $r$ is odd;}\\
(p_{n-r-1},-(q+1)/2) & \mbox{if $q$ is odd and $r$ is even;}\\
(p_{n-r-1},-l+(q+1)/2) & \mbox{if $q$ is odd and $r$ is odd.}
\end{cases}
$$
We need to show that $P_0,\ldots, P_{mn-1}$ is a permutation of elements of $G$. Suppose that $P_i=P_j$ for $i=qn+r$ and $j=pn+t$. If $r,t$ are both even or both odd, from $P_i=P_j$, we conclude that $p=q$. Thus, without loss of generality, suppose that $p$ is even and $q$ is odd. Then $p_t=p_{n-r-1}$ and so $t=n-r-1$ which implies that $t$ and $r$ are both even or both odd. If they are both odd, then $p/2=-l+(q+1)/2$ modulo $m$, and if they are both even, then $-l-p/2=-(q+1)/2$ modulo $m$. In either case we gave $p/2-(q+1)/2=-l \pmod m$, which is a contradiction since $1-l \leq p/2-(q+1)/2 \leq l-2$.

Next, we show that $g_i+g_{mn-i}=0$ for all $1\leq i \leq mn-1$. Let $i=qn+r$, where $0\leq q \leq n-1$ and $0\leq r \leq m-1$. So we can write $mn-i=(m-q-1)n+n-r$. Suppose $r\neq 0$. Then
$$g_i+g_{mn-i}=\left (h_r,(-1)^{q}l+2\lceil q/2 \rceil \right)+\left (h_{n-r},(-1)^{m-q-1}l+2\lceil (m-q-1)/2 \rceil \right).$$
Since $h_r+h_{n-r}=0$ and $m-1$ is even, this simplifies to
$$g_i+g_{mn-i}=(0,(-1)^q2l+2\lceil q/2 \rceil +2 \lceil (-q/2) \rceil-1)=(0,0) \in H \times \mathbb{Z}_m.$$
If $r=0$, then $g_i=(0,q)$ and one writes $mn-i=(m-q)n$. Therefore, $g_{mn-i}=(0,m-q)$ which again leads to $g_i+g_{mn-i}=(0,0)$. 

Finally, we claim that $P_i+P_{mn-i-1}= ((m_1-1)/2,\ldots, (m_d-1)/2)$ for all $i \in \{0,\ldots, mn-1\}$. We have $p_r+p_{m-r-1}=((m_1-1)/2,\ldots, (m_{d-1}-1)/2)$ for all $r=0,\ldots, m-1$ by the inductive hypothesis. Let $i=qn+r$, and so $mn-i-1=(m-q-1)n+n-r-1$. If $q$ is even and $r$ is odd, then
\begin{align}\nonumber
P_i+P_{mn-i-1} & =(p_r+p_{n-r-1},(m-1)/2) \\ \nonumber
& =((m_1-1)/2,\ldots, (m_{d-1}-1)/2,(m-1)/2).
\end{align}
The claim in other cases follows similarly.  
\end{proof}

\section{Central extensions}\label{extensions}

In this section, we describe a construction of DDP sequences via group extensions. Let $G$ be a group extension of $H$ by $N$ i.e., suppose that 
$1 \rightarrow N \rightarrow G \xrightarrow{\pi} H \rightarrow 1$ is a short exact sequence. We will describe an algorithm to \emph{lift} a DDP sequence 
in $H$ to $G$. By a lift of the DDP sequence $h_1,\ldots, h_{|H|}$ in $H$ to $G$, we mean a DDP sequence $g_1,\ldots, g_{|G|}$ such that $\pi(g_i)=h_i$ 
for $i=1,\ldots, |H|$. 

It turns out that in order for our algorithm of lifting a DDP sequence from $H$ to $G$ work, the group $N = \ker(\pi)$ must contain no \emph{real} elements of 
$G$ except the identity. 

\begin{definition}
An element $h\in G$ is said to be a {\it real} element of $G$ if there exists $g\in G$ such that $g^{-1}hg=h^{-1}$. We denote the set of real elements of 
$G$ by ${\cal R}(G)$. 
\end{definition}

Let $N$ be a normal subgroup of $G$. If the only real element of $G$ in $N$ is $1_G$ i.e., $N \cap {\cal R}(G)=\{1_G\}$, then
\begin{equation}\label{defreal}
\forall h \in N \backslash \{1_G\}~\forall g\in G: hgh \neq g,
\end{equation}
or equivalently, for abelian $N$,
$$\forall g \in G~\forall h_1,h_2\in N: h_1 \neq h_2 \Rightarrow h_1gh_1 \neq h_2gh_2.$$
If $N$ is contained in the center of $G$, then $N \cap {\cal R}(G)=\{1_G\}$ is equivalent to $N$ having odd order. 

\begin{theorem}\label{extconst}
Let $\pi: G \rightarrow H$ be an epimorphism such that $\ker (\pi)$ is an abelian group of odd order $m$ with $\ker(\pi) \cap {\cal R}(G)=\{1_G\}$. If 
$H$ is an even DDP group, then $G$ is an even DDP group. More precisely, let $p_0,\ldots, p_{n-1}$ be a DDP sequence in $H$. Then there exist at 
least $(2m)^{(n-1-e)/2}$ DDP sequences $P_0,\ldots, P_{mn-1}$ in $G$ such that $\pi(P_i)=p_i$ for all $i=0,\ldots, n-1$, where $e$ is the number of 
elements of order 2 in $H$. In particular 
$$|{\cal O}_G| \geq |{\cal O}_H |\times (2m)^{(n-1-e)/2}.$$
\end{theorem}
\begin{proof}
Let $p_0,\ldots, p_{n-1}$ be a DDP sequence in $H$. We define $h_0=1_H$ and $h_r=p_{r-1}^{-1}p_r$ for $1\leq r \leq n-1$. We define a bijection 
$\sigma: \{0,\ldots, n-1\} \rightarrow \{0,\ldots, n-1\}$ by letting $\sigma(r)$ to be the unique number in $\{0,\ldots, n-1\}$ such that $h_{\sigma(r)}=h_r^{-1}$. Let
$$
I=\{0\leq r \leq n-1: \sigma(r)=r\},
$$
and let $A$ be a set obtained by including exactly one of $r$ or $\sigma(r)$ for every $r \in \{0,\ldots, n-1\} \backslash I$, and define 
$B=\{0,\ldots, n-1\} \backslash (A \cup I)$. Clearly, $0\in I$ and there are $2^{(n-|I|)/2}$ choices for $A$.

Let also $\alpha_0,\alpha_1,\ldots, \alpha_{m-1}$ be a Slonimsky sequence in $N=\ker(\pi)$; such a special DDP sequence exists by Theorem \ref{spddp}, 
since $N$ has odd order. Let $\beta_0,\ldots, \beta_{m-1}$ be the sequence of signed differences. Let us denote the element 
$\alpha_{m-1}$ by $y_N$. By the definition of Slonimsky sequence, one has
\begin{align}
\alpha_i\alpha_{m-1-i}  = y_N = \alpha_{m-1}, ~ & \forall 0\leq i \leq m-1 \\
\beta_i\beta_{m-i}  =0_N, ~ &  \forall 1\leq i \leq m-1
\end{align}

In order to define the sequence $P_0,\ldots, P_{mn-1}$, we first define its sequence of consecutive differences $g_0,\ldots, g_{mn-1}$ 
as follows. For each $r\in A$, we let $g_r$ be an arbitrary element of $\pi^{-1}(h_r)$. Also, if $r\in A$, we define

\begin{equation}\label{onestar}
g_{\sigma(r)}=
\begin{cases}
g_r^{-1} & \mbox{if $r+\sigma(r)$ is odd;}\\
y_Ng_r^{-1}y_N & \mbox{if $r$ and $\sigma(r)$ are both odd;}\\
y^{-1}_Ng_r^{-1}y_N^{-1} & \mbox{if $r$ and $\sigma(r)$ are both even.}
\end{cases}
\end{equation}
To define $g_r$ for $r\in I$, choose $f_r \in \pi^{-1}(h_r)$ to be arbitrary. Then one can show that
$$
\pi^{-1}(h_r) = \{\alpha_if_r\alpha_i~|~ i\in\{1,\ldots ,m\} ~\mbox{and}~\alpha_i\in N\},
$$
and hence there exists $v_r\in N$ such that $f_r^{-1}=v_rf_rv_r$, since $\pi(f_r^{-1})=h_r=\pi(f_r)$. Then, choose $w_r\in N$ such that 
$w_r^2=v_ry_N^{(-1)^{r+1}}$, and let $g_r=w_rf_rw_r$. 
It follows from this definition that
$$
g_r^{-1}=\begin{cases} y_Ng_ry_N& \mbox{if $r \in I$ is even;} \\
y_N^{-1}g_ry_N^{-1} & \mbox{if $r \in I $ is odd}.
\end{cases}
$$

Next, we define $g_i$ for all $n\leq i \leq mn-1$. The idea is to present $g_0,\ldots ,g_{mn-1}$ as a union of $m$ blocks each 
containing $n$ elements so that $\pi$ 
maps each block onto $H$, alternating in increasing (for even blocks) or decreasing (for odd blocks) order of indices. To be more precise, by the Euclidean algorithm, there exist unique integers 
$0\leq r \leq n-1$ and $0\leq q \leq m-1$ such that $i=nq+r$. If $r=0$, let $g_{i}=\beta_q$. If $r \geq 1$, let 
\begin{equation}\label{twostar}
g_i=\begin{cases} \alpha_{q}g_{n-r}^{-1}\alpha_q & \mbox{if $q$ is odd and $r$ is odd};\\
\alpha_{q}^{-1}g_{n-r}^{-1}\alpha_q^{-1} & \mbox{if $q$ is odd and $r$ is even};\\
\alpha_{q}^{-1}g_{r}\alpha_q^{-1} & \mbox{if $q$ is even and $r$ is odd};\\
\alpha_{q}g_{r}\alpha_q & \mbox{if $q$ is even and $r$ is even}.
\end{cases} 
\end{equation}

We claim that the sequence $P_i=\prod_{k=0}^ig_k$, $0\leq i \leq mn-1$, is a DDP sequence. We prove by induction on $0\leq i \leq mn-1$ that for $i=nq+r$, 
we have
\begin{equation}\label{partials}
P_i=\begin{cases} P_{n-r-1}\alpha_q & \mbox{if $q$ is odd and $r$ is odd};\\
P_{n-r-1}\alpha_q^{-1} & \mbox{if $q$ is odd and $r$ is even};\\
P_{r}\alpha_q^{-1} & \mbox{if $q$ is even and $r$ is odd};\\
P_{r}\alpha_q & \mbox{if $q$ is even and $r$ is even}.
\end{cases} 
\end{equation}
The claim is clearly true for all $0\leq i \leq n-1$. Suppose the claim is true for $i=nq+r$. Suppose that $q$ and $r$ are both odd. 
The proof in all other cases is similar. If $r=n-1$ then
$$
P_{i+1}=P_i  g_{i+1}=(P_0 \alpha_q)\beta_{q+1}=P_0\alpha_{q+1}
$$
as claimed. If $0\leq r<n-1$. Then $i+1=nq+(r+1)$ and we have
$$
P_{i+1}=P_i g_{i+1}=(P_{n-r-1}\alpha_q)(\alpha_{q}^{-1}g_{n-r-1}^{-1}\alpha_q^{-1})=P_{n-r-2}\alpha_q^{-1}
$$
as claimed. It follows from \eqref{partials} that $P_i \neq P_j$ for $0\leq i<j \leq mn-1$. To see this, suppose $P_i=P_j$ for $i=nq_1+r_1$ and $j=nq_2 +r_2$. 
Suppose that $q_1$ and $q_2$ are even. The proof in other cases is similar. Then $p_{r_1}=\pi(P_i)=\pi(P_j)=p_{r_2}$ which implies that $r_1=r_2=r$. But then 
$\alpha_{q_1}=(P_r^{-1}P_i)^{\pm 1}=(P_r^{-1} P_j)^{\pm 1}=\alpha_{q_2}$, and so $q_1=q_2$, hence $i=j$. 

Next, we show that $g_i \neq g_j$ for all $0\leq i<j \leq mn-1$. On the contrary, suppose that $g_i=g_j$ for $i=qn+r$ and $j=pn+s$ where $1\leq r,s < n$. 
There are two cases:

Case 1. $p \equiv q \pmod 2$. If $p,q$ are both even, then $h_r=\pi(g_i)=\pi(g_j)=h_s$, and if $p,q$ are both odd, then 
$h_{n-r}=\pi(g_i)^{-1}=\pi(g_j)^{-1}=h_{n-s}$. In either case, we conclude that $r=s$. If $r$ is even, this implies that 
$\alpha_p g_r \alpha_p=\alpha_q g_r \alpha_q$ (if $p,q$ are even) or $\alpha_q^{-1}g_{n-r}^{-1}\alpha_q^{-1}=\alpha_p^{-1}g_{n-r}^{-1}\alpha_p^{-1}$ 
(if $p,q$ are odd). In either case, since $N\cap{\cal R}(G) = \{1_G\}$, we must have $p=q$, and so $i=j$. 

Case 2. Without loss of generality, suppose $q$ is even and $p$ is odd. Then 
$\alpha_q^{\pm 1} g_r \alpha_q^{\pm 1} = \alpha_p^{\pm 1} g_{n-s}^{-1} \alpha_p^{\pm 1}$. 
By projecting onto $H$ via $\pi$, we must have $h_r=h_{n-s}^{-1}$. If $r=n-s \in I$, then $r$ and $s$ are both even or both odd. If they are both even, 
it follows from $g_i = g_j$ that $\alpha_q g_r \alpha_q=\alpha_p^{-1}g_r^{-1}\alpha_p^{-1}$ which implies that $\alpha_p\alpha_q = y_N$, 
which is a contradiction, since $p$ and $q$ have different parity. If $r$ and $s$ are both odd, then 
$\alpha_q^{-1} g_r \alpha_q^{-1}=\alpha_pg_r^{-1}\alpha_p$ which leads to the same contradiction. 

Thus, suppose $r \in A \cup B$. Without loss of generality, suppose $r\in A$, and so $n-s=\sigma(r) \in B$. 
If both $r$ and $s$ are odd, according to Eq.\ \eqref{onestar} we have 
$\alpha_q^{-1}g_r\alpha_q^{-1} = \alpha_pg_{\sigma(r)}^{-1}\alpha_p = \alpha_py_N^{-1}g_r y_N^{-1}\alpha_p$, which implies 
$\alpha_p\alpha_q = y_N$, a contradiction. Similarly, if $r$ and $s$ are both even, we have 
$\alpha_qg_r\alpha_q = \alpha_p^{-1}g_{\sigma(r)}^{-1}\alpha_p^{-1} = \alpha_p^{-1}y_Ng_ry_N \alpha_p^{-1}$, which again implies 
$\alpha_p\alpha_q = y_N$, a contradiction. If $r$ is odd and $\sigma(r)$ is even, then $\alpha_q^{-1}g_r\alpha_q^{-1}=\alpha_p^{-1}g_{\sigma(r)}^{-1}\alpha_p^{-1}=\alpha_p^{-1}g_r\alpha_p^{-1}$ which implies that $\alpha_p=\alpha_q$, a contradiction. Finally, if $r$ is even and $\sigma(r)$ is odd, then $\alpha_qg_r\alpha_q=\alpha_pg_{\sigma(r)}^{-1}\alpha_p=\alpha_p g_r \alpha_p$ which implies that $\alpha_q=\alpha_p$, a contradiction. 

We have shown that $P_0,\ldots, P_{mn-1}$ is a DDP sequence in $G$ with $\pi(P_i)=p_i$ for all $0\leq i \leq n-1$. Recall that in constructing 
the set $A$, we have two choices per each pair $(r, \sigma(r))$. Moreover, for each $r\in A$, we have $m$ choices in defining $g_r$. It follows that 
there are at least $(2m)^{|A|}=(2m)^{(n-|I|)/2}$ DDP sequences which are lifts of a given DDP sequence in $H$. Since $I$ is comprised of $1_H$ and 
elements of order 2, each DDP sequence in $H$ has at least $(2m)^{(n-e-1)/2}$ lifts to $G$, where $e$ is the number of elements of order 2 in $H$.
\end{proof}

\begin{cor}\label{centralext}
Every central extension of an even DDP group by an odd abelian group is a DDP group. 
\end{cor}

\begin{proof}
Let $N$ be an odd abelian group and $H$ be an even DDP group. Suppose that $\pi: G \rightarrow H$ is an epimorphism with $\ker(\pi) \cong N$. We need to show that $G$ is a DDP group. Since $\ker (\pi)$ is an odd abelian group and, by the definition of central extension, the normal subgroup $\ker(\pi)$ lies in the center of $G$, one has $\ker(\pi) \cap {\cal R}_G =\{1_G\}$, the conditions of Theorem \ref{extconst} hold, hence $G$ is an even DDP group. 
\end{proof}

\begin{theorem}\label{nil2}
Let $G$ be a finite odd nilpotent group and $K$ be an even DDP group. Then $G \times K$ is a DDP group. 
\end{theorem}

\begin{proof}
Let $Z_0 \lhd Z_1 \lhd \cdots \lhd Z_n=G$ be the upper central series of $G$. We prove by a finite reverse induction on $0\leq i \leq n$ that $(G/Z_i) \times K$ is a DDP group. The claim is clearly true for $i=n$. Suppose we have proved that $(G/Z_{i+1}) \times K$ is DDP for $0\leq  i < n$ and we show that $(G/Z_i) \times K$ is DDP. Consider the epimorphism
$$
\pi_i: \dfrac{G}{Z_{i}} \times K \rightarrow \dfrac{G}{Z_{i+1}} \times K, ~ \pi_i(g + Z_i,~k) := (g + Z_{i+1},~k) 
$$
induced by the inclusion $Z_i \hookrightarrow Z_{i+1}$. By the inductive hypothesis $G/(Z_{i+1})\times K$ is DDP. Moreover, $\ker (\pi_i) \cong (Z_{i+1}/Z_i) \times \{1_K\}$ which is contained in the center of $G/Z_i \times K$. It then follows from Corollary \ref{centralext} that $(G/Z_i)\times K$ is DDP. When $i=0$, we conclude that $G\times K$ is DDP.
\end{proof}

\section{The abelian case}\label{abelian}
In this section, we determine all finite abelian DDP groups. We begin with describing an obstruction to the existence of a DDP 
sequence in the abelian case. For an abelian group $G$, we use $0_G$ (or simply $0$) to denote its identity element.
\begin{lemma}\label{none}
If $G$ is an abelian DDP group, then it has a unique element of order 2. 
\end{lemma}
\begin{proof}
Let $x_1,\ldots, x_k$ be a DDP sequence in $G$. Then we have 
\begin{equation}
\label{sum}
-x_1+x_k=\sum_{i=1}^{k-1} -x_i+x_{i+1}=\sum_{g \in G}g.
\end{equation}
Now let us assume to the contrary that either $G$ has odd order or it has more than one element of order 2. Firstly, if $G$ has odd order, 
we have $2\sum_{g\in G}g=\sum_{g\in G}g+\sum_{g\in G}(-g)=0_G$, and (\ref{sum}) implies that $x_k = x_1$, which is not allowed.
Secondly, if $G$ has more than one element of order 2, then one can write $G=\mathbb{Z}_m \times \mathbb{Z}_n \times H$ for even 
integers $m,n$, and an abelian group $H$. But then
$$
\sum_{g\in G}g=\left (mn|H|/2, mn|H|/2, mn\sum_{h\in H}h \right ) = (0_{\mathbb{Z}_m}, 0_{\mathbb{Z}_n}, 0_H) = 0_G \in G,
$$
since $\sum_{i \in \mathbb{Z}_n}i=n(n-1)/2=n/2$ modulo $n$ and $2\sum_{h\in H}h=0_H$. 
Now it follows again from (\ref{sum}) that
$$
-x_1+x_k = 0_G,
$$
which contradicts the assumption that $x_1,\ldots, x_k$ are distinct. 
\end{proof}

In the next Lemma we consider the group $(\mathbb{Z}_n,+)$ where $n=2^m$. 

\begin{lemma}\label{power2}
Let $n=2^m$, where $m$ is a positive natural number. Then the following statements hold.
\begin{itemize}
\item[a)] The sequence $x_i=i(i+1)/2$, $0\leq i \leq n-1$, is a DDP sequence modulo $n$ for all $m\geq 1$. 

\item[b)] The sequence 
$$y_i=\begin{cases}i(i+1)/2 & \mbox{if $0\leq i< 2^{m-2}$ or $ 3\cdot 2^{m-2} \leq i < 2^{m}$,}\\
i(i+1)/2+2^{m-1} & \mbox{if $2^{m-2} \leq i < 3 \cdot 2^{m-2}$,}
\end{cases}$$
is a DDP sequence modulo $n$ for all $m\geq 2$. 
\end{itemize}
\end{lemma}

\begin{proof}
Since $x_{i+1}-x_i=i+1$, part (a) is equivalent to the claim that $i \mapsto i(i+1)/2$ is a bijection on $\mathbb{Z}_n$. If $n=2^m$, then the map $i \mapsto i(i+1)/2$ is a bijection modulo $n$. To see this, let $j \in \mathbb{Z}_n$ be arbitrary. Then $8j+1$ is a quadratic residue modulo $2^{m+3}$ \cite[Thm.\ 5-1]{quadres2}. Hence there exists $i \in \mathbb{Z}_n$ such that $8j+1=(2i+1)^2 \pmod 2^{m+3}$, and so $j \equiv i(i+1)/2 \pmod n$. It follows that $i \mapsto i(i+1)/2$ is onto, hence a bijection, on $\mathbb{Z}_n$. 

For part (b), one verifies that the sequence of consecutive differences of $y_0,\ldots, y_{n-1}$ is given by
$$
0,1,2,\ldots, 2^{m-2}-1, 3 \cdot 2^{m-2}, 2^{m-2}+1, \ldots, 3 \cdot 2^{m-2}-1, 2^{m-2}, 3 \cdot 2^{m-2}+1, \ldots, 2^m-1,
$$
which is obtained from the sequence $0,1,\ldots, 2^m-1$ by exchanging $2^{m-2}$ and the product $3 \cdot 2^{m-2}$, hence $y_0,\ldots, y_{n-1}$ is a DDP sequence. 
\end{proof}

\begin{cor}
If $n=2^m$, $m\geq 3$, then $|{\cal O}_{\mathbb{Z}_n}|\geq n$. 
\end{cor}

\begin{proof}
For $m\geq 3$, the two DDP sequences in Lemma \ref{power2} are distinct. Moreover, $rx_0,\ldots, rx_{n-1}$, and $ry_0,\ldots, ry_{n-1}$, are DDP sequences for every odd number $r \in \mathbb{Z}_n$, and the corollary follows. 
\end{proof}

\begin{theorem}
Let $G$ be an abelian group. Then $G$ is a DDP group if and only if $G$ has exactly one element of order 2. 
\end{theorem}

\begin{proof}
In light of Lemma \ref{none}, it is left to show that if $G=H \times \mathbb{Z}_{2^m}$, where $m\geq 1$ and $H$ is an odd abelian group, then $G$ is DDP. Since $H$ is an odd nilpotent group and $\mathbb{Z}_{2^m}$ is an even DDP group by Lemma \ref{power2}, the claim follows from Theorem \ref{nil2}.
\end{proof}

\begin{cor}\label{sizeo}
Let $c_1=1$, $c_2=2$, and $c_m=2^m$ for $m\geq 3$. If 
$G=\mathbb{Z}_{2^m} \times \mathbb{Z}_{n_1}\times \cdots \times \mathbb{Z}_{n_k}$ 
where $n_1,\ldots, n_k$ are odd integers and $m\geq 1$, then 
$$
|{\cal O}_G| \geq c_m \times (2n_1)^{2^{m-1}-1}\times (2n_2)^{2^{m-1}n_1-1}\cdots (2n_k)^{2^{m-1}n_1\cdots n_{k-1}-1}.
$$
In particular, if an abelian group $G$ has size $2^m kl $, where $m\geq 1$ and $k,l$ are relatively prime odd integers, then $|{\cal O}_G| \geq (2k)^{l-1}.$
\end{cor}

\begin{proof}
Proof is by induction on $k$. If $k=0$, the claim follows from Lemma \ref{power2}. For the inductive step, let 
$G=\mathbb{Z}_{n_{k+1}}\times H$, where by the inductive hypothesis
$$|{\cal O}_H| \geq c_m \times (2n_1)^{2^{m-1}-1}\times (2n_2)^{2^{m-1}n_1-1}\cdots (2n_k)^{2^{m-1}n_1\cdots n_{k-1}-1}.$$
By Theorem \ref{extconst}, we have
$$|{\cal O}_G| \geq (2n_{k+1})^{(|H|-1-e)/2}|{\cal O}_H|,$$
where $e$ is the number of elements of order 2. It follows from $G=\mathbb{Z}_{2^m} \times \mathbb{Z}_{n_1}\times \cdots \times \mathbb{Z}_{n_k}$ that one has $e=1$, and the claim follows. 
The last claim of the Corollary \ref{sizeo} follows from $G \cong \mathbb{Z}_{2^m} \times \mathbb{Z}_k \times \mathbb{Z}_l$.
\end{proof}

\section{The non abelian case}\label{nonabelian}

{Computer searches show that the smallest non abelian DDP group is the dihedral group $D_5$, which has 320 DDP sequences. If we present $D_5$ in terms of generators and relations as
$$
D_5 \cong \langle a,b ~|~ a^5=b^2=1,~aba=b \rangle,
$$ 
an example of a DDP sequence in $D_5$ is
$$
1,a,a^3,ba^3,a^2,b, a^4, ba^4, ba^2, ba,
$$
with the corresponding sequence of distinct divisors
$$
1,a,a^2,ba,b,ba^2,ba^4,ba^3,a^3,a^4.
$$
The group $D_6$ has 3072 DDP sequences, and the alternating group on four elements $A_4$ has 2304 DDP sequences. 

Computer searches 
also confirm that $D_7$ is a DDP group, and we conjecture that $D_n$ is a DDP group for all $n\geq 5$.
As we noted in Lemma \ref{none}, an abelian group of odd order is not DDP. 
However, the next example shows that in the non abelian case, DDP groups of odd order do exist. 

\begin{example}
Consider the smallest non abelian group of an odd order, that is let $G=\mathbb{Z}_7 \rtimes \mathbb{Z}_3$ be 
the non abelian group of order 21. In generators and relations, $G$ is given by
$$
G \cong \langle a,b ~|~ a^7=b^3=1,~a^2b=ba \rangle.
$$
The following sequence is a DDP sequence in $G$:
$$1, a,ba^6,ba^2,a^3,a^5,b,b^2a^4,ba^4,b^2a^2,ba^5,ba^3,a^6,b^2a^3,ba,b^2,b^2a^6,a^2,b^2a,b^2a^5,a^4,$$
where the sequence of distinct divisors is given by
$$1, a,ba^2,a^3,b^2a^6,a^2,ba,ba^4,b^2a^3,b,b^2a,a^5,b^2,b^2a^5,b^2a^2,ba^3,a^6,ba^6,b^2a^4,a^4,ba^5.$$
\end{example}

The next lemma provides a construction of DDP groups via semidirect products. Consider for example the 
semidirect product $G=\mathbb{Z}_9 \rtimes_\phi \mathbb{Z}_6$, where 
$\phi: \mathbb{Z}_6 \rightarrow Aut(\mathbb{Z}_9)$ is defined by 

$$
\phi_t(j):=
\begin{cases}
j & \mbox{if}~ t\equiv 0 \pmod 3;\\
4j & \mbox{if}~ t\equiv 1 \pmod 3;\\
7j & \mbox{if}~ t\equiv 2 \pmod 3.\\
\end{cases}
$$ 
\noindent Then $G$ is a DDP group by the following lemma.

\begin{lemma}\label{semidirect}
Let $\phi: \mathbb{Z}_{n} \rightarrow Aut(\mathbb{Z}_m)$ be a group homomorphism such that $1+\phi_s(1)$ is a generator of 
$\mathbb{Z}_m$ for all $s\in \mathbb{Z}_n$. If $m$ is odd and $n$ is even, then
$\mathbb{Z}_m \rtimes_\phi \mathbb{Z}_n$ is a DDP group. 
\end{lemma}

\begin{proof}
Consider the projection $\pi: \mathbb{Z}_m \rtimes_\phi \mathbb{Z}_n \rightarrow \mathbb{Z}_n$ with $\ker(\pi)=\mathbb{Z}_m \times \{0\}$. The claim follows from Theorem \ref{extconst} if we show that $\alpha g \alpha=g \Rightarrow \alpha=0$ for all $\alpha \in \mathbb{Z}_m \times \{0\}$ and $g\in \mathbb{Z}_m \rtimes_\phi \mathbb{Z}_n$. Let $g=(r,s)$ and $\alpha=(k,0)$. Then
$$\alpha g \alpha=(k,0)(r,s)(k,0)=(r+k+\phi_s(k),s) \neq (r,s),$$
since $k+ \phi_s(k) \neq 0$ for all $k\neq 0$; otherwise, $k(1+\phi_s(1))=0$ which contradicts the assumption. 
\end{proof}

Finally, we show that there exist infinitely many non abelian DDP groups.
\begin{theorem}\label{infinite}
Let $p$ be a prime with $p \equiv 3 \pmod 4$ and let $t$ be a primitive root modulo $p$. Then 
$\mathbb{Z}_p \rtimes_\phi \mathbb{Z}_{p-1}$ is a DDP group, where $\phi: \mathbb{Z}_{p-1} \rightarrow Aut(\mathbb{Z}_p)$ 
is given by $\phi_s(x)=t^{2s}x$. In particular, there exist infinitely many non abelian DDP groups. 
\end{theorem}

\begin{proof}
We first show that $t^{2s}$ is not congruent to $-1$ modulo $p$ for every $s\in \mathbb{Z}_{p-1}$. If on the contrary, $t^{2s} \equiv -1 \pmod p$, we have $4s \equiv 0 \pmod{p-1}$, which implies that $2s \equiv 0 \pmod{p-1}$ since $p \equiv 3 \pmod 4$. But then $t^{2s} \equiv 1 \pmod p$, which is a contradiction. It follows that $1+\phi_s(1) \neq 0$ for all $s\in \mathbb{Z}_{p-1}$, and the claim 
follows from Lemma \ref{semidirect}. 
\end{proof}


\begin{thebibliography}{20}

\bibitem{BattenSane}L. M. Batten and S. Sane, Permutations with a distinct difference property, {\it Discrete Math.} {\bf 261} (2003), 59--67.

\bibitem{Ferentz} S. Bauer-Mendelberg and M. Ferentz, On Eleven-Interval Twelve-Tone Rows, 
{\it Perspectives of New Music} {\bf 3} (1965), 93--103.

\bibitem{Costas1}J. P. Costas, A study of a class of detection waveforms having nearly ideal range-Doppler ambiguity properties, {\it Proc. IEEE}, Vol. 72, 1984, pp. 996--1009.

\bibitem{Costas2}J. P. Costas, Medium constraints on sonar design and performance, Tech. Rep. Class 1 Rep. R65EMH33, General Electric Company, Fairfield, CT, USA, 1965

\bibitem{DISJ}K. Drakakis, F. Iorio, S. Rickard, and J. Walsh, Results of the enumeration of Costas arrays of order 29, {\it Adv. Math. Commun.} {\bf 5} (2011), 547--553.

\bibitem{Eimert} H. Eimert, {\it Lehrbuch der zw\"{o}lftontechnik. Weisbaden}, Breitkopf und H\"{a}rtel, 1952.

\bibitem{Gilbert} E. N. Gilbert, Latin squares which contain no repeated diagrams, {\it SIAM Review.} {\bf 7} (1965), 189--198.

\bibitem{Golomb1}S. W. Golomb, Algebraic constructions for Costas arrays, J. Combin. Theory (A) {\bf 37} (1984), 13--21.

\bibitem{Golomb2} S. W. Golomb, Construction of signals with favourable correlation properties, in: {\it A. Pott et al. (Eds.), Difference Sets, Sequences and their Correlation Properties}, Kluwer, Dordrecht, 1999, pp. 159--194.

\bibitem{Golomb3} S. W. Golomb and H. Taylor, Constructions and properties of Costas arrays, {\it Proc. IEEE}, Vol. 72, 1984, pp. 1143--1163.

\bibitem{OEIS}M. Gustar, Number of difference sets for permutations of $[2n]$ with distinct differences, The on-line encyclopedia of integer sequences, \url{https://oeis.org/A141599} (2008). 

\bibitem{Klein} F. H. Klein, Die Grenze der Halbtonwelt, Die Musik {\bf 17} (1925), 281--286.

\bibitem{quadres2}W. J. LeVeque, {\it Topics in Number Theory, Volumes I and II}, Courier Corporation, 2012. 

\bibitem{Slonimsky} N. Slonimsky, \emph{Thesaurus of Scales and Melodic Patterns}, Amsco Publications, 8th edition, New York, 1975.

\end{thebibliography}
\end{document}